\documentclass[12pt,a4paper]{article}
\usepackage{CJK}
\usepackage{indentfirst}
\usepackage[dvips]{graphics}

\usepackage{amsmath,amssymb,amsfonts,amsthm,graphics,CJK}
\usepackage{makeidx}
\begin{CJK}{GBK}{song}
\newtheorem{theorem}{Theorem}[section]

\newtheorem{proposition}{Proposition}[section]
\newtheorem{lemma}{Lemma}[section]

\newtheorem{corollary}{Corollary}[section]

\newtheorem{remark}{Remark}[section]
\newtheorem{observation}{Observation}[section]

\newtheorem{example}{Example}[section]
\usepackage{color}
\begin{document}
\title{The generalized 3-connectivity of\\ Cartesian product
graphs\footnote{Supported by NSFC No. 11071130.}}
\author{\small Hengzhe Li, Xueliang Li, Yuefang Sun
\\
\small Center for Combinatorics and LPMC-TJKLC
\\
\small Nankai
University, Tianjin 300071, China
\\
\small lhz2010@mail.nankai.edu.cn; lxl@nankai.edu.cn;
syf@cfc.nankai.edu.cn}
\date{}
\maketitle

\begin{abstract}
The generalized connectivity of a graph, which was introduced
recently by Chartrand et al., is a generalization of the concept of
vertex connectivity. Let $S$ be a nonempty set of vertices of $G$, a
collection $\{T_1,T_2,\ldots,T_r\}$ of trees in $G$ is said to be
internally disjoint trees connecting $S$ if $E(T_i)\cap
E(T_j)=\emptyset$ and $V(T_i)\cap V(T_j)=S$ for any pair of distinct
integers $i,j$, where $1\leq i,j\leq r$. For an integer $k$ with
$2\leq k\leq n$, the $k$-connectivity $\kappa_k(G)$ of $G$ is the
greatest positive integer $r$ for which $G$ contains at least $r$
internally disjoint trees connecting $S$ for any set $S$ of $k$
vertices of $G$. Obviously, $\kappa_2(G)=\kappa(G)$ is the
connectivity of $G$. Sabidussi showed that $\kappa(G\Box H) \geq
\kappa(G)+\kappa(H)$ for any two connected graphs $G$ and $H$. In
this paper, we first study the $3$-connectivity of the Cartesian
product of a graph $G$ and a tree $T$, and show that $(i)$ if
$\kappa_3(G)=\kappa(G)\geq 1$, then $\kappa_3(G\Box T)\geq
\kappa_3(G)$; $(ii)$ if $1\leq \kappa_3(G)< \kappa(G)$, then
$\kappa_3(G\Box T)\geq \kappa_3(G)+1$. Furthermore, for any two
connected graphs $G$ and $H$ with $\kappa_3(G)\geq\kappa_3(H)$, if
$\kappa(G)>\kappa_3(G)$, then $\kappa_3(G\Box H)\geq
\kappa_3(G)+\kappa_3(H)$; if $\kappa(G)=\kappa_3(G)$, then
$\kappa_3(G\Box H)\geq \kappa_3(G)+\kappa_3(H)-1$. Our result could
be seen as a generalization of Sabidussi's result. Moreover, all the
bounds are sharp.


\end{abstract}

\section{Introduction}

All graphs in this paper are undirected, finite and simple. We refer
to the book \cite{bondy} for graph theoretic notations and
terminology not described here. Let $G$ be a graph, the connectivity
$\kappa(G)$ of a graph $G$ is defined as $min\{|S|\,|\, S\subseteq
V(G)\ and\ G-S\ is\ disconnected\ or\ trivial\}$. Whitney
\cite{whitney} showed an equivalent definition of the connectivity
of a graph. For each pair of vertices $x,y$ of $G$, let
$\kappa(x,y)$ denote the maximum number of internally disjoint paths
connecting $x$ and $y$ in $G$. Then the connectivity $\kappa(G)$ of
$G$ is $min\{\kappa(x,y)\,|\, x,y\ are\ distinct\ vertices\ of G\}$.

The Cartesian product of graphs is an important method to construct
a bigger graph, and plays a key role in design and analysis of
networks. In the past several decades, many authors have studied the
(edge) connectivity of the Cartesian product graphs. For example,
Sabidussi derived the following result about the connectivity of
Cartesian product graphs.

\begin{theorem}{\upshape \cite{sabidussi}}
Let $G$ and $H$ be two connected graphs. Then $\kappa(G\Box H) \geq
\kappa(G)+\kappa(H)$.
\end{theorem}

More information about the (edge) connectivity of the Cartesian
product graphs can be found in \cite{chiue,imrich, kla,sabidussi,xu}.

The generalized connectivity of a graph $G$, which was introduced
recently by Chartrand et al. in \cite{chartrand}, is a natural and
nice generalization of the concept of vertex connectivity. A tree
$T$ is called an {\em $S$-tree} ({\em
$\{u_1,u_2,\ldots,u_k\}$-tree}) if $S\subseteq V(T)$, where
$S=\{u_1,u_2,\ldots,u_k\}\in V(G)$. A family of trees $T_1,
T_2,\ldots, T_r$ are {\em internally disjoint $S$-trees} if
$E(T_i)\cap E(T_j)=\emptyset$ and $V(T_i)\cap V(T_j)=S$ for any pair
of integers $i$ and $j$, where $1\leq i<j\leq r$. We use $\kappa(S)$
to denote the greatest number of internally disjoint $S$-trees. For
an integer $k$ with $2\leq k\leq n$, the $k$-$connectivity\
\kappa_k(G)$ of $G$ is defined as $min\{\kappa(S)\,|\, S\in V(G)\
and\ |S|=k\}$. Clearly, when $|S|=2$, $\kappa_2(G)$ is nothing new
but the connectivity $\kappa(G)$ of $G$, that is,
$\kappa_2(G)=\kappa(G)$, which is the reason why one addresses
$\kappa_k(G)$ as the generalized connectivity of $G$. By convention,
for a connected graph $G$ with less than $k$ vertices, we set
$\kappa_k(G)=1$. For any graph $G$, clearly, $\kappa(G)\geq 1 $ if
and only if $\kappa_3(G)\geq 1$.

In addition to being a natural combinatorial measure, the
generalized connectivity can be motivated by its interesting
interpretation in practice. For example, suppose that $G$ represents
a network. If one considers to connect a pair of vertices of $G$,
then a path is used to connect them. However, if one wants to
connect a set $S$ of vertices of $G$ with $|S|\geq 3$, then a tree
has to be used to connect them. This kind of tree for connecting a
set of vertices is usually called a Steiner tree, and popularly used
in the physical design of VLSI, see \cite{Sherwani}. Usually, one
wants to consider how tough a network can be, for the connection of
a set of vertices. Then, the number of totally independent ways to
connect them is a measure for this purpose. The generalized
$k$-connectivity can serve for measuring the capability of a network
$G$ to connect any $k$ vertices in $G$.

In {\upshape\cite{li2}}, Li and Li investigated the complexity of
determining the generalized connectivity and derived that for any
fixed integer $k\geq 2$, given a graph $G$ and a subset $S$ of
$V(G)$, deciding whether there are $k$ internally disjoint trees
connecting $S$, namely deciding whether $\kappa(S)\geq k$, is
$NP$-complete.

Chartrand et al. {\upshape\cite{chartrand}} got the following result
for complete graphs.

\begin{theorem}{\upshape\cite{chartrand}}
For every two integers $n$ and $k$ with $2\leq k\leq n$,
$\kappa_k(K_n)=n-\lceil k/2\rceil$.
\end{theorem}

Okamoto and Zhang \cite{OZ} investigated the generalized
connectivity for regular complete bipartite graphs $K_{a,a}$.
Recently, Li et al. \cite{li1} got the following result for general
complete bipartite graphs.

\begin{theorem}{\upshape\cite{li1}}
Given any two positive integers $a$ and $b$, let $K_{a, b}$ denote a
complete bipartite graph with a bipartition of sizes $a$ and $b$,
respectively. Then we have the following results: if $k>b-a+2$ and
$a-b+k$ is odd then
$$
\kappa_{k}(K_{a,b})=\frac{a+b-k+1}{2}+\lfloor\frac{(a-b+k-1)(b-a+k-1)}{4(k-1)}\rfloor;
$$
if $k>b-a+2$ and $a-b+k$ is even then
$$
\kappa_{k}(K_{a,b})=\frac{a+b-k}{2}+\lfloor\frac{(a-b+k)(b-a+k)}{4(k-1)}\rfloor;
$$
and if $k\leq b-a+2$ then
$$
\kappa_{k}(K_{a,b})=a.
$$
\end{theorem}

Li et al. {\upshape\cite{li3}} got the following upper bounds of
$\kappa_3(G)$ for general graphs.

\begin{theorem} {\upshape\cite{li3}} Let $G$ be a connected
graph with at least three vertices. If $G$ has two adjacent vertices
with minimum degree $\delta$, then $\kappa_3(G)\leq \delta-1$.
\end{theorem}

\begin{theorem}{\upshape\cite{li3}}
Let $G$ be a connected graph with $n$ vertices. Then,
$\kappa_3(G)\leq \kappa(G).$ Moreover, the upper bound is sharp.
\end{theorem}

In this paper, we study the $3$-connectivity of Cartesian product
graphs. The paper is organized as follows. In Section~$2$, we recall
the definition and properties of Cartesian product graphs, and give
some basic results about the internally disjoint $S$-trees. In
Sections~$3$ and $4$, we study the $3$-connectivity of the Cartesian
product of a graph $G$ and a tree $T$, and show that $(i)$ if
$\kappa_3(G)=\kappa(G)\geq 1$, then $\kappa_3(G\Box T)\geq
\kappa_3(G)$; $(ii)$ if $1\leq \kappa_3(G)< \kappa(G)$, then
$\kappa_3(G\Box T)\geq \kappa_3(G)+1$. Moreover, the bounds are
sharp. As a consequence, we get that $\kappa_3(Q_n)=n-1$, where
$Q_n$ is the $n$-hypercube. In Section~$5$, we study the
$3$-connectivity of the Cartesian product of two connected graphs
$G$ and $H$, and show that for any two connected graphs $G$ and $H$
with $\kappa_3(G)\geq\kappa_3(H)$, if $\kappa(G)>\kappa_3(G)$, then
$\kappa_3(G\Box H)\geq \kappa_3(G)+\kappa_3(H)$; if
$\kappa(G)=\kappa_3(G)$, then $\kappa_3(G\Box H)\geq
\kappa_3(G)+\kappa_3(H)-1$. Moreover, all the bounds are sharp. Our
result could be seen as a generalization of Theorem~$1.1$.

\section{Some basic results}

We use $P_n$ to denote a path with $n$ vertices. A path $P$ is
called a {\em u-v path}, denoted by $P_{u,v}$, if $u$ and $v$ are
the endpoints of $P$.

Recall that the $Cartesian\linebreak[2]\ product$ (also called the
{\em square product}) of two graphs $G$ and $H$, written as $G\Box
H$, is the graph with vertex set $V(G)\times V(H)$, in which two
vertices $(u,v)$ and $(u',v')$ are adjacent if and only if $u=u'$
and $(v,v')\in E(H)$, or $v=v'$ and $(u,u')\in E(G)$. Clearly, the
Cartesian product is commutative, that is, $G\Box H\cong H\Box G$.
The edge $(u,v)(u',v')$ is called {\em one-type edge} if $(u,u')\in
E(G)$ and $v=v'$; similarly, the $(u,v)(u',v')$ is called {\em
two-type edge} if $u=u'$ and $(v,v')\in E(H)$.

Let $G$ and $H$ be two graphs with $V(G)=\{u_1,u_2,\ldots,u_n\}$ and
$V(H)=\{v_1,v_2,\ldots,v_m\}$, respectively. We use $G(u_j,v_i)$ to
denote the subgraph of $G\Box H$ induced by the set
$\{(u_j,v_i)\,|\,1\leq j\leq n\}$. Similarly, we use $H(u_j,v_i)$ to
denote the subgraph of $G\Box H$ induced by the set
$\{(u_j,v_i)\,|\,1\leq i\leq m\}$. It is easy to see
$G(u_{j_1},v_i)=G(u_{j_2},v_i)$ for different $u_{j_1}$ and
$u_{j_2}$ of $G$. Thus, we can replace $G(u_{j},v_i)$ by $G(v_i)$
for simplicity. Similarly, we can replace $H(u_{j},v_i)$ by
$H(u_j)$. For $x=(u,v)$, we refer to $(u,v')$ and $(u',v)$ as $the\
vertices\ corresponding\ to\ x$ in $G(v')\ (\,=G(u,v')\,)$ and
$H(u')\ (\,=H(u',v)\,)$, respectively. Similarly, we can define the
path and tree corresponding to some path and tree, respectively.

Imrich and Klav\v{z}ar gave the following result in \cite{imrich}.

\begin{proposition}{\upshape \cite{imrich}}
The Cartesian product of two graphs $G$ and $H$ is connected if and
only if both graphs $G$ and $H$ are connected.
\end{proposition}

By Proposition $2.1$, we only consider the generalized connectivity
$\kappa_3(G)$ of the Cartesian product of two connected graphs.

\begin{proposition}{\upshape \cite{imrich}}
The Cartesian product is associative, that is, $(G_1\Box G_2)\Box
G_3\cong G_1\Box (G_2\Box G_3)$.
\end{proposition}

In order to show our main results, we need the following well-known
result.

\begin{theorem}[{\upshape Menger's Theorem \cite{bondy}}] Let $G$ be a
$k$-connected graph, and let $x$ and $y$ be a pair of distinct
vertices in $G$. Then there exist $k$ internally disjoint paths
$P_1,P_2,\ldots, P_k$ in $G$ connecting $x$ and $y$.
\end{theorem}

Let $G$ be a connected graph, and $S=\{x_1,x_2,x_3\}\subseteq V(G)$.
We first have the following observation about internally disjoint
$S$-trees.

\begin{observation}
Let $G$ be a connected graph, $S=\{x_1,x_2,x_3\}\subseteq V(G)$, and
$T$ be an $S$-tree. Then there exists a subtree $T'$ of $T$ such
that $T'$ is also an $S$-tree such that $1\leq d_{T'}(x_i)\leq 2,\
|\{x_i\,|\, d_{T'}(x_i)=1\}|\geq 2$ and
$\{x\,|\,d_{T'}(x)=1\}\subseteq S$. Moreover, if $|\{x_i\,|\,
d_{T'}(x_i)=1\}|=3$, then all the vertices of $V(T')\setminus
\{x_1,x_2,x_3\}$ have degree 2 except for one vertex, say $x$ with
$d_{T'}(x)=3$; if there exists one vertex of $S$, say $x_1$, has
degree 2 in $T'$, then $T'$ is an $x_2$-$x_3$ path.
\end{observation}
\begin{proof}
It is easy to check that this observation holds by deleting vertices
and edges of $T$.
\end{proof}

\begin{remark}
$(i)$ Since the path between any two distinct vertices is unique in
$T$, the tree $T'$ obtained from $T$ is unique in Observation $2.1$.
Such a tree is called a minimal $S$-tree (or minimal
$\{x_1,x_2,x_3\}$-tree).

$(ii)$ Let $S=\{x,y,z\}\subseteq V(G)$. Throughout this paper, we
can assume that each $S$-tree is a minimal $S$-tree.
\end{remark}
\begin{lemma} Let $G$ be a graph with $\kappa_3(G)=k\geq 2$,
$S=\{x,y,z\}\subseteq V(G)$. Then, we have the following result.

$(i)$ If $G[S]$ is a clique, then there exist $k$ internally
disjoint $S$-trees $T_1,T_2,\ldots, T_k$, such that $E(T_i)\cap
E(G[S])=\emptyset$ for $1\leq i\leq k-2$.

$(ii)$ If $G[S]$ is not a clique, then there exist $k$ internally
disjoint $S$-trees $T_1,T_2,\ldots, T_k$, such that $E(T_i)\cap
E(G[S])=\emptyset$ for $1\leq i\leq k-1$.

\end{lemma}
\begin{proof}
We first prove $(i)$. Clearly, by the definition of $S$-trees, we
know $|\{T_i\, |\, E(T_i)\cap E(G[S])\neq \emptyset\}|\leq 3$. Let
$\{T_1,T_2,\ldots, T_k\}$ be $k$ internally disjoint $S$-trees. If
$|\{T_i\,|\, E(T_i)\cap E(G[S])\neq\emptyset\}|\leq 2$, we are done
by exchanging subscript. Thus, suppose $|\{T_i\,|\, E(T_i)\cap
E(G[S])\neq\emptyset\}|=3$. Without loss of generality, we assume
$E(T_i)\cap E(G[S])\neq\emptyset$, where $i=k-2,k-1,k$. It is easy
to check that $T_{k-2},T_{k-1},T_k$ must have the structures as
shown in Figures $1a$ and $1b$. But, for these two cases, we can
obtain $T_{k-2}',T_{k-1}', T_k'$ from $T_{k-2},T_{k-1}, T_k$, such
that $E(T_{k-2}')\cap \{xy,xz,yz\}=\emptyset$. See Figs. $1c.$ and
$1d$, where the tree $T_{k-2}'$ is shown by dotted lines. Thus
$T_1,T_2,\ldots, T_{k-3},T_{k-2}',T_{k-1}',T_k'$ are our desired
$S$-trees.

The proof of $(ii)$ is similar to that of $(i)$, and thus is
omitted.
\end{proof}

\begin{figure}[h,t,b,p]
\begin{center}
\scalebox{0.6}[0.6]{\includegraphics{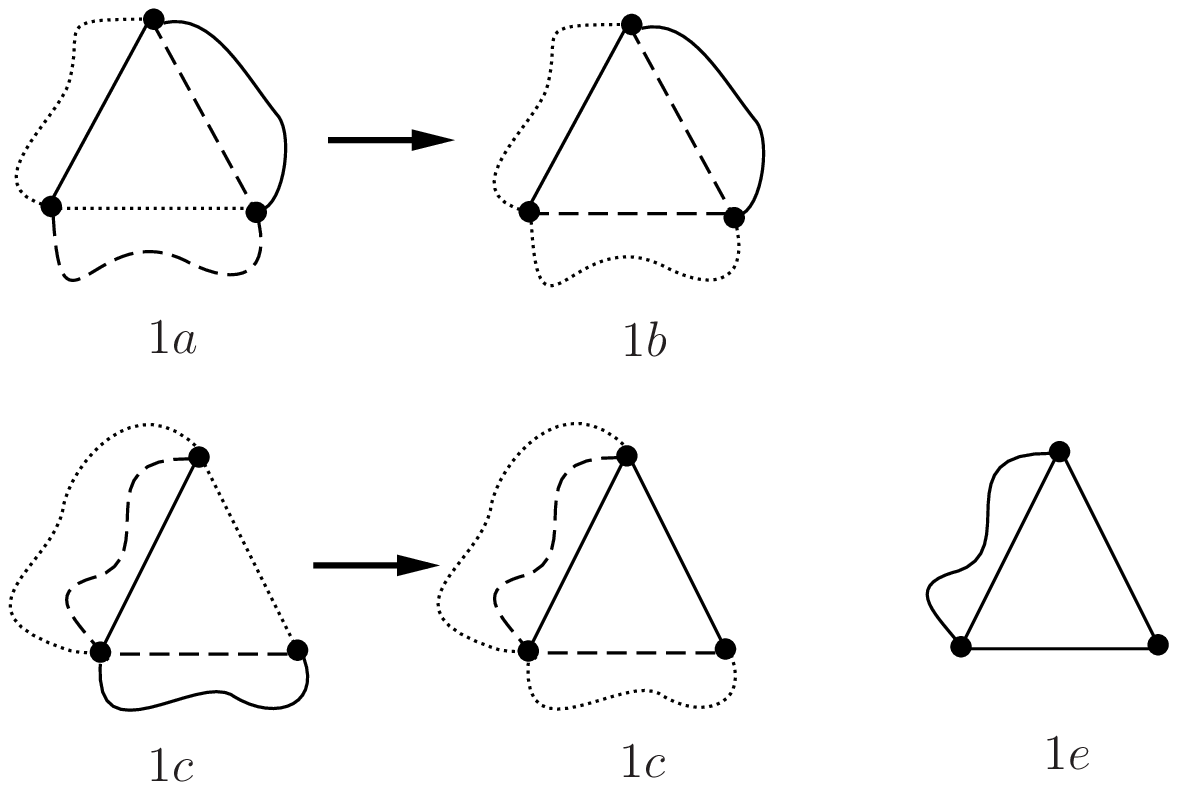}}

Figure 1. $T_{k-2}',T_{k-1}',T_k'$. An edge is shown by a straight
line.\\ The edges (or paths) of a tree are shown by the same type of
lines.
\end{center}
\end{figure}
\begin{remark}
Let $G$ be a graph with $\kappa_3(G)=k\geq 2$, $S=\{x,y,z\}\subseteq
V(G)$. If $|\{E(T_i)\,|\,E(T_i)\cap E(G[S])\neq\emptyset\}|\geq 2$
for any collection $\mathcal{T}$ of $k$ internally disjoint
$S$-trees, then $G[S]$ is a clique. Moreover, $T_{k-1}\cup T_k$ must
have the structure as shown in Figure~$1e$.
\end{remark}

\section{The Cartesian product of a connected graph and a path}

In this section, we show the following theorem.
\begin{theorem} Let $G$ be a graph and $P_m$ be a path with $m$
vertices. We have the following results.

$(i)$ If $\kappa_3(G)=\kappa(G)\geq 1$, then $\kappa_3(G\Box
P_m)\geq \kappa_3(G)$. Moreover, the bound is sharp.

$(ii)$ If $1\leq \kappa_3(G)< \kappa(G)$, then $\kappa_3(G\Box
P_m)\geq \kappa_3(G)+1$. Moreover, the bound is sharp.
\end{theorem}

We shall prove Theorem $3.1$ by a series of lemmas. Since the proofs
of $(i)$ and $(ii)$ are similar, we only show $(ii)$. Let $G$ be a
graph with $V(G)=\{u_1,u_2,\ldots,u_n\}$ such that $1\leq
\kappa_3(G)< \kappa(G)$, $V(P_m)=\{v_1,v_2,\ldots,v_m\}$ such that
$v_i$ and $v_j$ are adjacent if and only if $|i-j|=1$.

Set $\kappa_3(G)=k$ for simplicity. To prove $(ii)$, it suffices to
prove that for any $S=\{x,y,z\}\subseteq V(G\Box H)$, there exist
$k+1$ internally disjoint $S$-trees. We proceed our proof by the
following three lemmas.

\begin{lemma} If $x,y,z$ belongs to the same $V(G(v_i))$, $1\leq i\leq m$, then there
exist $k+1$ internally disjoint $S$-trees.
\end{lemma}
\begin{proof}
Without loss of generality, we assume $x,y,z\in V(G(v_1))$. Since
$\kappa_3(G)=k$, there exist $k$ internally disjoint $S$-trees
$T_1,T_2,\ldots, T_k$ in $G(v_1)$. We need another $S$-tree
$T_{k+1}$ such that $T_{k+1}$ and $T_i$ are internally disjoint,
where $i=1,2,\ldots,k$. Let $x',y',z'$ be the vertices corresponding
to $x,y,z$ in $G(v_2)$, and $T_1'$ be the tree corresponding to
$T_1$ in $G(v_2)$. Therefore, The tree $T_{k+1}$ obtained from
$T_1'$ by adding three edges $xx', yy',zz'$ is a desired tree.
\end{proof}

\begin{lemma} If exact two of $x,y,z$ are contained in some $G(v_i)$, then there
exist $k+1$ internally disjoint $S$-trees.
\end{lemma}
\begin{proof}

We may assume $x,y\in V(G(v_1)),z\in V(G(v_2))$. In the following
argument, we can see that this assumption has no influence on the
correctness of our proof. Let $x',y'$ be the vertices corresponding
to $x,y$ in $G(v_2)$, $z'$ be the vertex corresponding to $z$ in
$G(v_1)$. Consider the following two cases.

\begin{figure}[h,t,b,p]
\begin{center}
\includegraphics{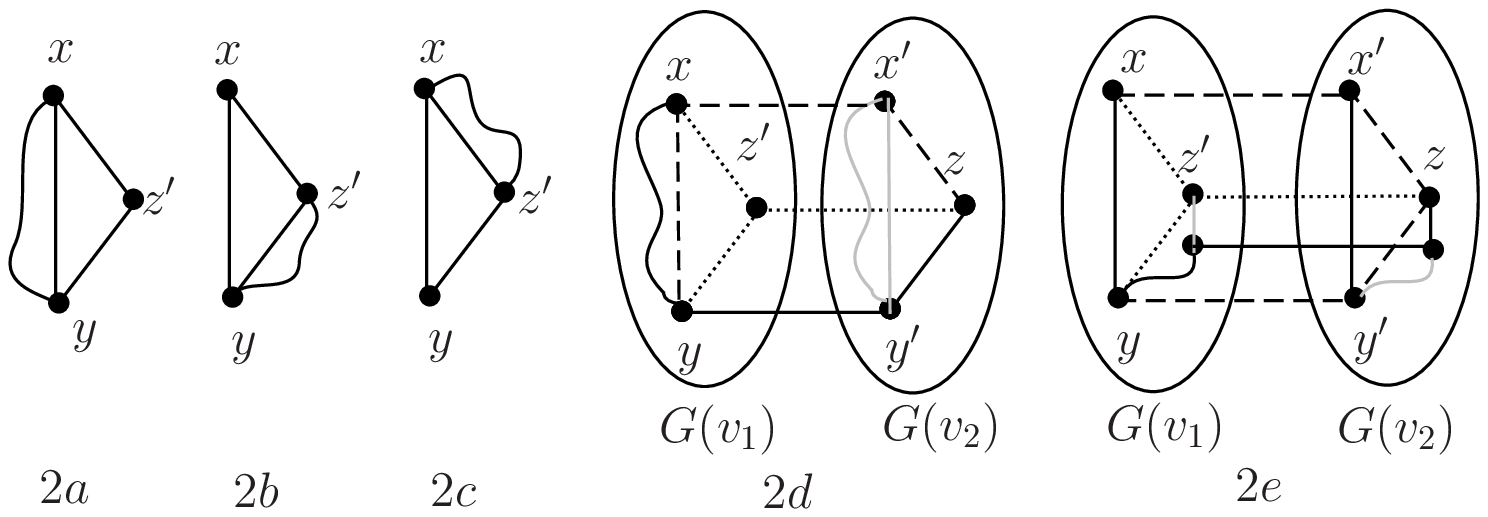}

Figure~2. The edges (or paths) of a tree are shown by the same type
of lines.\\ The lightest lines stand for edges (or paths) not
contained in $T_i^*$.
\end{center}
\end{figure}

{\flushleft\textbf{Case 1.}}\quad $z'\not\in \{x,y\}$.

Let $S'=\{x,y,z'\}$, and $T_1,T_2,\ldots T_k$ be $k$ internally
disjoint $S'$-trees in $G(v_1)$ such that $|\{T_i\,|\,E(T_i)\cap
E(G(v_1)[S']\}\neq\emptyset|$ is as small as possible. We can assume
that $E(T_i)\cap E(G(v_1)[S'])=\emptyset$ for each $i$, where $1\leq
i\leq k-2$ by Lemma~$2.1$.

{\flushleft\textbf{Subcase 1.1.}} $E(T_i)\cap
E(G(v_1)[S'])=\emptyset$ for each $i$, where $1\leq i\leq k$.

Assume that $d_{T_i}(z')=1$ for $1\leq i\leq k_1$ and
$d_{T_i}(z')=2$ for $k_1+1\leq i\leq k$. If $k_1=k$, let $T_i^*$ be
the tree obtained from $T_i$ by adding $z_iz_i'$ and $z_i'z$, and
deleting $z'$, where $1\leq i\leq k-1$ and $z_i$ is the only
neighbor of $z'$ in $T_i$, and $Z_i'$ is the vertex corresponding to
$z_i$ in $G(v_2)$. Let $T_k^*=T_k+zz'$ and $T_{k+1}^*=T_k'+xx'+yy'$,
where $T_k'$ is the tree in $G(v_2)$ corresponding to $T_k$. Thus
$T_1^*,T_2^*,\ldots,T_{k+1}^*$ are $k+1$ internally disjoint
$S$-trees.

Now suppose $k_1<k$. For $1\leq i\leq k_1$, we construct $T_i^*$
similar to the above procedure. For $k_1+1\leq i\leq k-1$, let
$T_i^*$ be the tree obtained from $T_i$ by adding
$z_{i,1}z_{i,1}',z_{i,1}'z,z_{i,2}z_{i,2}'$ and $z_{i,2}'z$ and
deleting $z'$, where $k_1+1\leq i\leq k$,
$N_{T_i}(z')=\{z_{i,1},z_{i,2}\}$, and $z_{i,1}'$ and $z_{i,2}'$ are
the vertices corresponding to $z_{i,1}$ and $z_{i,2}$ in $G(v_2)$,
respectively. Let $T_k^*=T_k+zz'$ and $T_{k+1}^*=T_k'+xx'+yy'$,
where $T_k'$ is the tree in $G(v_2)$ corresponding to $T_k$. Thus
$T_1^*,T_2^*,\ldots,T_{k+1}^*$ are $k+1$ internally disjoint
$S$-trees.

{\flushleft\textbf{Subcase 1.2.}} $E(T_i)\cap
E(G(v_1)[S'])\neq\emptyset$ for some $i$, where $i=k-1,k$.

For a tree $T_i$ with $E(T_i)\cap E(G(v_1)[S'])=\emptyset$, we can
construct ${T_i}^*$ similar to that of Subcase $1.1$.

If $E(T_{k-1})\cap E(G(v_1)[S'])=\emptyset$ and $E(T_k)\cap
E(G(v_1)[S'])\neq\emptyset$, say $y'z\in E(T_k)\cap E(G(v_1)[S'])$.
Let $T_k^*=T_k+zz'$ and $T_{k+1}^*=T_k'+xx'+yy'$, where $T_k'$ is
the tree corresponding to $T_k$ in $G(v_2)$.

If $E(T_{k-1})\cap E(G(v_1)[S'])\neq\emptyset$ and $E(T_k)\cap
E(G(v_1)[S'])\neq\emptyset$. Then $T_{k-1}\cup T_{k}$ must have one
of the structures as shown in Figures $2a, \ 2b$ and $2c$ by
Remark~$2.2$. If $T_{k-1}$ and $T_{k}$ have the structures as shown
in Figure $2a$, then we can obtain trees $T_{k-1}^*,T_{k}^*$ and
$T_{k+1}^*$ as shown in Figure $2d$. If $T_{k-1}$ and $T_{k}$ have
the structures as shown in Figure $2b$, then we can obtain trees
$T_{k-1}^*,T_{k}^*$ and $T_{k+1}^*$ as shown in Figure $2e$. If
$T_{k-1}$ and $T_{k}$ have the structures as shown in Figure $2c$,
then we can obtain trees $T_{k-1}^*,T_{k}^*$ and $T_{k+1}^*$ similar
to those in Figure $2d$.

{\flushleft\textbf{Case 2.}}\quad $z'\in \{x,y\}$.

Without loss of generality, assume $z'=y$. Since
$\kappa(G)>\kappa_3(G)=k$, by Menger's Theorem, there exist at least
$k+1$ internally disjoint {\em x-y} paths $P^1,P^2,\ldots,P^{k+1}$.
Assume that $y_i$ is the only neighbor of $y$ in $P^i$, and that
$y_i'$ is the vertex corresponding to $y_i$ in $G(v_2)$. If $x$ and
$y$ are nonadjacent in $P^i$, let $T_{i}$ be the tree obtained from
$P^i$ by adding $y_iy_i'$ and $y_i'z$. If $x$ and $y$ are adjacent
in $P^i$, let $T_{i}$ be the tree obtained from $P^i$ by adding
$yz$. Since $G$ is a simple graph, there exists at most one path
$P^i$ such that $x$ and $y$ are adjacent on $P^i$. Thus $T_{i},1\leq
i \leq k+1,$ are $k+1$ internally disjoint $S$-trees.
\end{proof}

\begin{lemma} If $x,y,z$ are contained in distinct
$G(v_i)$s, then there exist $k+1$ internally disjoint $S$-trees.
\end{lemma}
\begin{proof}

We may assume that $x\in V(G(v_1)),y\in V(G(v_2)),z\in V(G(v_3))$.
In the following argument, we can see that this assumption has no
influence on the correctness of our proof. Let $y',z'$ be the
vertices corresponding to $y,z$ in $G(v_1)$, $x',z''$ be the
vertices corresponding to $x,z$ in $G(v_2)$ and $x'',y''$ be the
vertices corresponding to $x,y$ in $G(v_3)$. We consider the
following three cases.

{\flushleft\textbf{Case 1.}} $x,y',z'$ are distinct vertices in
$G(v_1)$

Let $S'=\{x,y',z'\}$, and $T_1,T_2,\ldots T_k$ be $k$ internally
disjoint $S'$-trees in $G(v_1)$ such that $|\{T_i\,|\,E(T_i)\cap
E(G(v_1)[S'])\neq\emptyset\}|$ is as small as possible. We can
assume that $E(T_i)\cap E(G(v_1)[S'])=\emptyset$ for each $i$, where
$1\leq i\leq k-2$ by Lemma~$2.1$. For each $T_i$ such that
$E(T_i)\cap E(G(v_1)[S'])=\emptyset$, we can obtain an $S$-tree
$T_i^*$ from $T_i$ similar to that in Subcase~$1.1$ of Lemma~$3.2$.

If $E(T_{k-1})\cap E(G(v_1)[S'])=\emptyset$ or $E(T_{k-1})\cap
E(G(v_1)[S'])=\emptyset$. Without loss of generality, we assume
$E(T_{k-1})\cap E(G(v_1)[S'])=\emptyset$. Let $T_k^*$ be the tree
obtained from $T_k$ by adding edges $y'y,z'z''$ and $z''z$,
$T_{k+1}^*$ be the tree obtained from $T_k''$ by adding $x''x',x'x$
and $y''y$, where $T_k''$ is the tree corresponding to $T_k$ in
$G(v_3)$. Thus, $T_i^*$s, $1\leq i\leq k+1$, are $k$ internally
disjoint $S$-tree.

Otherwise, that is, $E(T_{k-1})\cap E(G(v_1)[S'])\neq\emptyset$ and
$E(T_k)\cap E(G(v_1)[S'])\neq\emptyset$. Then $T_{k-1}$ and $T_{k}$
must have the structures as shown in Figures $3a, \ 3b$ and $3c$. If
$T_{k-1}$ and $T_{k}$ have the structures as shown in Figure~$3a$,
then we can obtain trees $T_{k-1}^*,T_{k}^*$ and $T_{k+1}^*$ as
shown in Figure~$3d$. If $T_{k-1}$ and $T_{k}$ have the structures
as shown in Figure~$3b$, then we can obtain trees
$T_{k-1}^*,T_{k}^*$ and $T_{k+1}^*$ as shown in Figure~$3e$. If
$T_{k-1}$ and $T_{k}$ have the structures as shown in Figure~$3c$,
then we can obtain trees $T_{k-1}^*,T_{k}^*$ and $T_{k+1}^*$ as
shown in Figure~$3f$.

\begin{figure}[h,t,b,p]
\begin{center}
\scalebox{0.7}[0.6]{\includegraphics{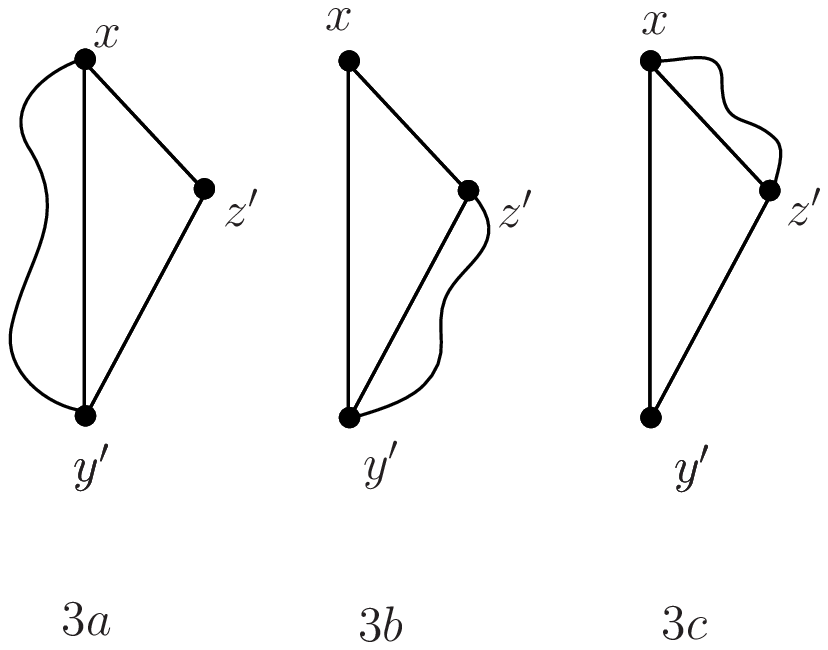}}
\scalebox{0.7}[0.6]{\includegraphics{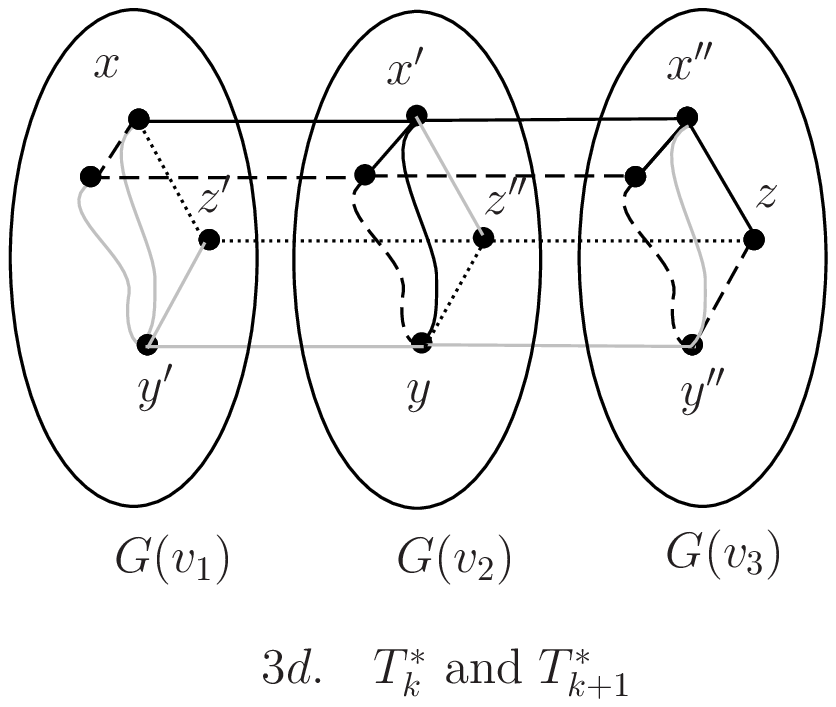}}
\scalebox{0.7}[0.6]{\includegraphics{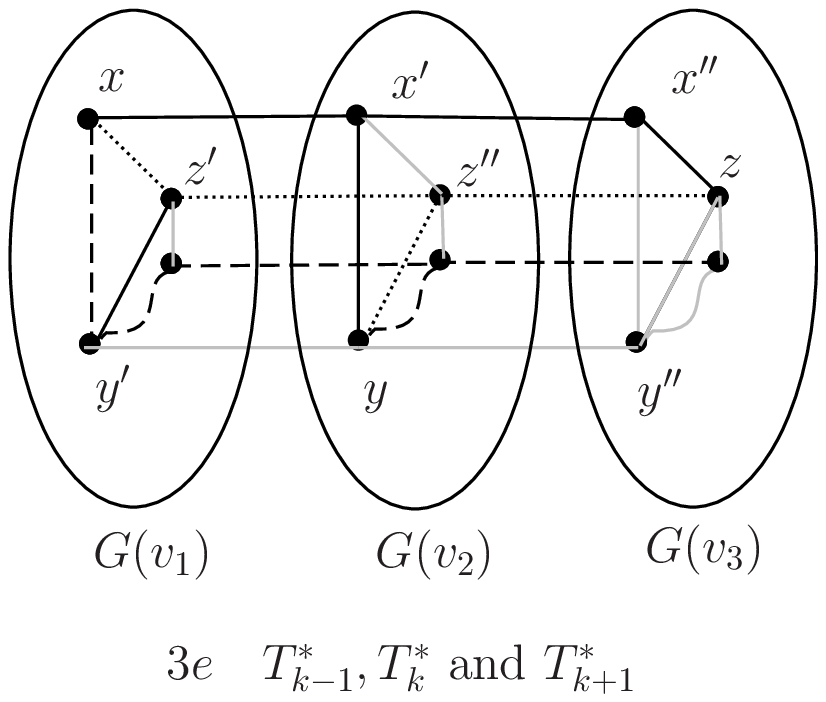}}
\scalebox{0.7}[0.6]{\includegraphics{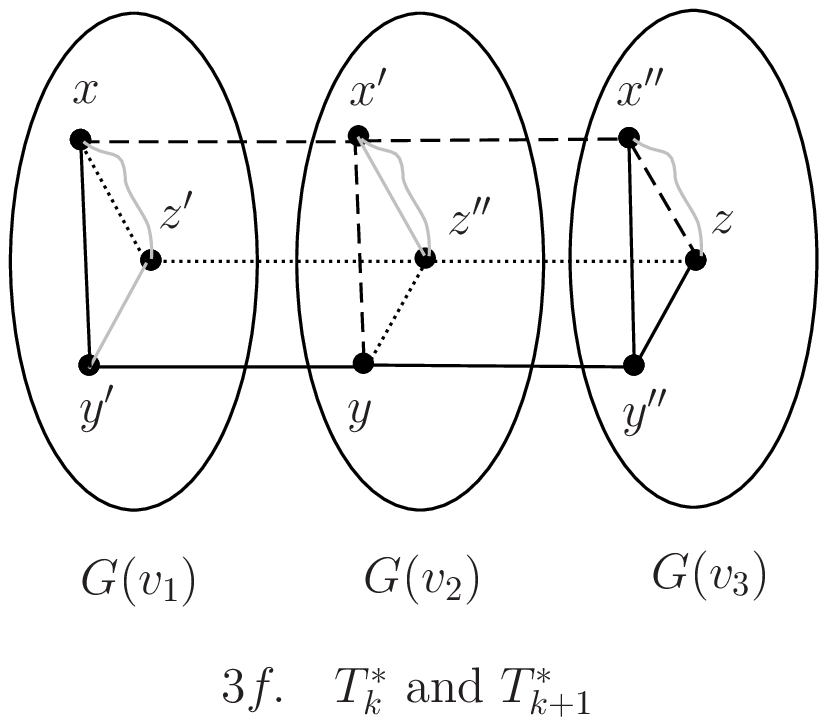}}
\scalebox{0.7}[0.6]{\includegraphics{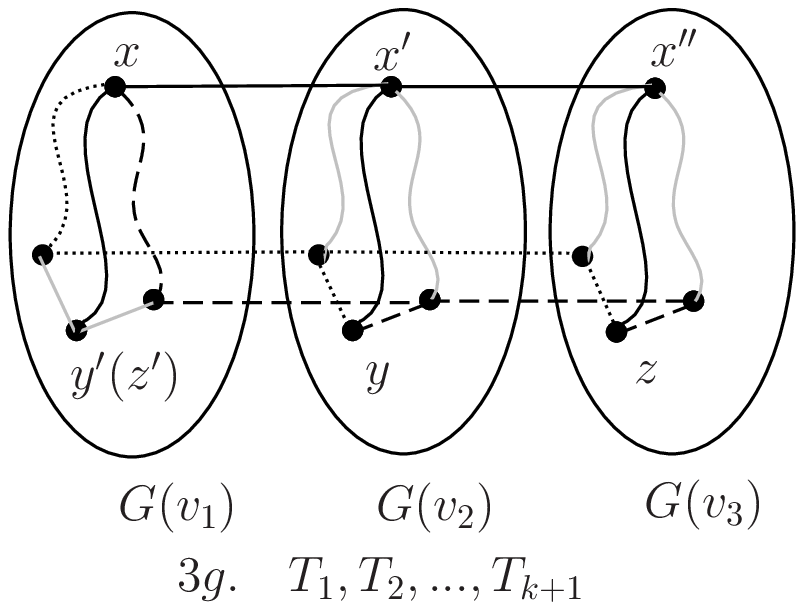}}
\scalebox{0.7}[0.6]{\includegraphics{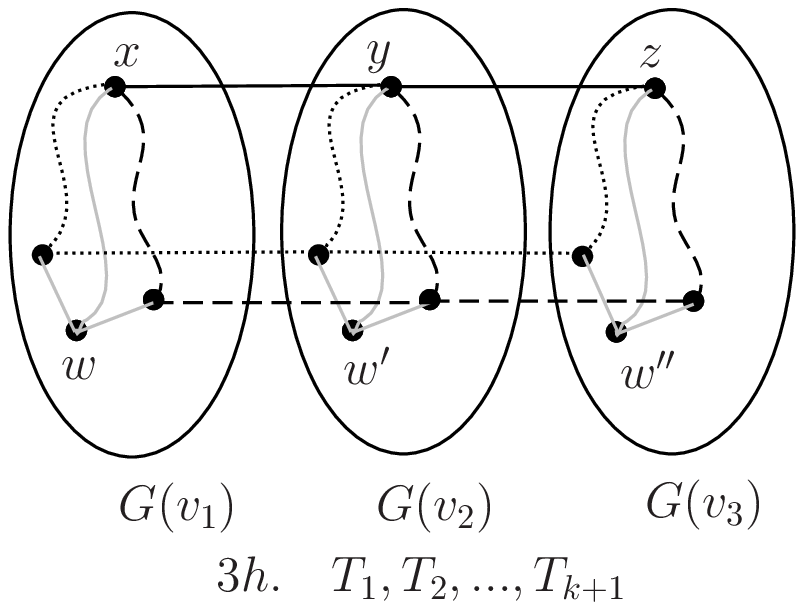}}

Figure~3. The edges (or paths) of a tree are shown by the same type
of lines.\\ The lightest lines stand for edges (or paths) not
contained in $T_i^*$.
\end{center}
\end{figure}

{\flushleft\textbf{Case 2.}} Two of $x, y',z'$ are the same vertex
in $G(v_1)$.

If $y'=z'$, since $\kappa(G)>\kappa_3(G)=k$, by Menger's Theorem, it
is easy to construct $k+1$ internally disjoint $S$-trees. See Figure
$3g$. The other cases ($x=y'$ or $x=z'$) can be proved with similar
arguments.

{\flushleft\textbf{Case 3.}} $x,y',z'$ are the same vertex in
$G(v_1)$.

Since $\kappa(G)>\kappa_3(G)=k$, by Menger's Theorem, it is easy to
construct $k+1$ internally disjoint $S$-trees. See Figure $3h$.
\end{proof}

We have the following observation by the argument in the proof of
Theorem~$3.1$.

\begin{observation} The $k+1$ internally disjoint $S$-trees
consist of three kinds of edges --- the edges of original trees (or
paths), the edges corresponding the edges of original trees (or
paths) and two-type edges.
\end{observation}

Note that $Q_{n}\cong P_{2}\Box P_{2}\Box\cdots\Box P_{2}$, where
$Q_n$ is the $n$-hypercube. We have the following corollary.

\begin{corollary} Let $Q_n$ be the $n$-hypercube with $n\geq 2$. Then
$\kappa_3(Q_n)=n-1$.
\end{corollary}

\begin{proof}
It is easy to check that $\kappa_3(Q_2)=1$. Assume that the result
holds for $\kappa_3(Q_{n-1}),n\geq 3$. By Theorem $3.1$,
$\kappa_3(Q_{n})\geq n-1$. On the other hand, since $Q_n$ is
$n$-regular, $\kappa_3(Q_{n})\leq n-1$ by Theorem $1.3$. Thus
$\kappa_3(Q_n)=n-1$.
\end{proof}

\begin{example}
Let $H_1$ and $H_2$ be two complete graphs of order $n$, and let
$V(H_1)=\{u_1,u_2,\ldots,u_n\}, V(H_2)=\{v_1,v_2,\ldots,v_n\}$. We
now construct a graph $G$ as follows:

\centerline{$V(G)=V(H_1)\cup V(H_2) \cup \{w\}$, where $w$ is a new
vertex;}

\centerline{$E(G)=E(H_1)\cup E(H_2)\cup \{u_iv_j\,|\, 1\leq i,j \leq
n\} \cup \{wu_i\,|\,1\leq i \leq n\}$}.

It is easy to check that $\kappa_3(G\Box K_2)=\kappa_3(G)=n$ by
Theorems $1.2$ and $1.5$.
\end{example}
\begin{remark}
We know that the bounds of $(i)$ and $(ii)$ in Theorem $3.1$ are
sharp by Example $3.1$ and Corollary~$3.1$.
\end{remark}

\section{The Cartesian product of a connected graph and a tree}

\begin{theorem} Let $G$ be a connected graph and $T$ be a tree.
We have the following result.

$(i)$ If $\kappa_3(G)=\kappa(G)\geq 1$, then $\kappa_3(G\Box T)\geq
\kappa_3(G)$. Moreover, the bound is sharp.

$(ii)$ If $1\leq \kappa_3(G)< \kappa(G)$, then $\kappa_3(G\Box
T)\geq \kappa_3(G)+1$. Moreover, the bound is sharp.
\end{theorem}

We shall prove Theorem $3.1$ by a series of lemmas. Since the proofs
of $(i)$ and $(ii)$ are similar, we only show $(ii)$. It suffices to
show that for any $S=\{x,y,z\}\subseteq G\Box H$, there exist $k+1$
internally disjoint $S$-trees. Set $\kappa_3(G)=k$,
$V(G)=\{u_1,u_2,\ldots,u_n\}$, and $V(T)=\{v_1,v_2,\ldots,v_m\}$.

Let $x\in V(G(v_i)),y\in V(G(v_j)),z\in V(G(v_k))$ be three distinct
vertices. If there exists a path in $T$ containing $v_i, v_j$ and
$v_k$, then we are done from Theorem $3.1$. If $i, j$ and $k$ are
not distinct integers, such a path must exist. Thus, suppose that
$i, j$ and $k$ are distinct integers, and that there exists no path
containing $v_i, v_j$ and $v_k$. By Observation $2.1$, there exists
a tree $T$ in $H$ such that $d_{T}(v_i)=d_{T}(v_j)=d_{T}(v_k)=1$ and
all the vertices of $V(T)\setminus \{v_i,v_j,v_k\}$ have degree 2
except for one vertex, say $v_4$ with $d_{T}(v_4)=3$. Without loss
of generality, we set $i=1,j=2,k=3$, $S'=\{x',y',z'\}$, where $x',
y'$ and $z'$ are the vertices corresponding to $x,y$ and $z$ in
$G(v_4)$, respectively. Furthermore, we assume $v_iv_4\in E(T)$,
where $1\leq i\leq 3$. In the following argument, we can see that
this assumption has no influence on the correctness of our proof. We
proceed our proof by the following three lemmas.

\begin{lemma} If $x', y'$ and $z'$ are three distinct
vertices, then there exist $k+1$ internally disjoint $S$-trees.
\end{lemma}
\begin{proof}
Let $T_1,T_2,\ldots T_k$ be $k$ internally disjoint $S'$-trees in
$G(v_4)$ such that $|\{T_i\,|\,E(T_i)\cap
E(G(v_4)[S'])\neq\emptyset\}|$ is as small as possible. We can
assume $E(T_i)\cap E(G(v_4)[S'])=\emptyset$ for $1\leq i\leq k-2$ by
Lemma $2.1$.

\begin{figure}[h,t,b,p]
\begin{center}
\scalebox{0.8}[0.7]{\includegraphics{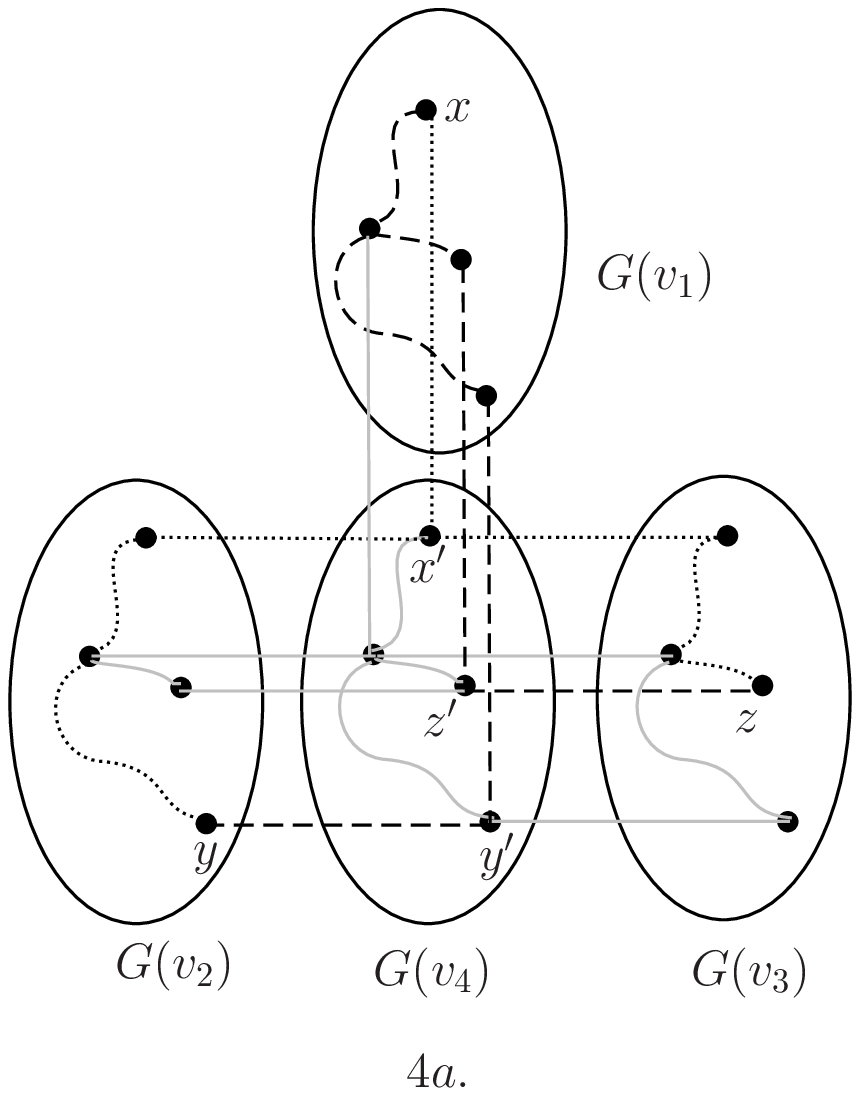}}
\scalebox{0.8}[0.7]{\includegraphics{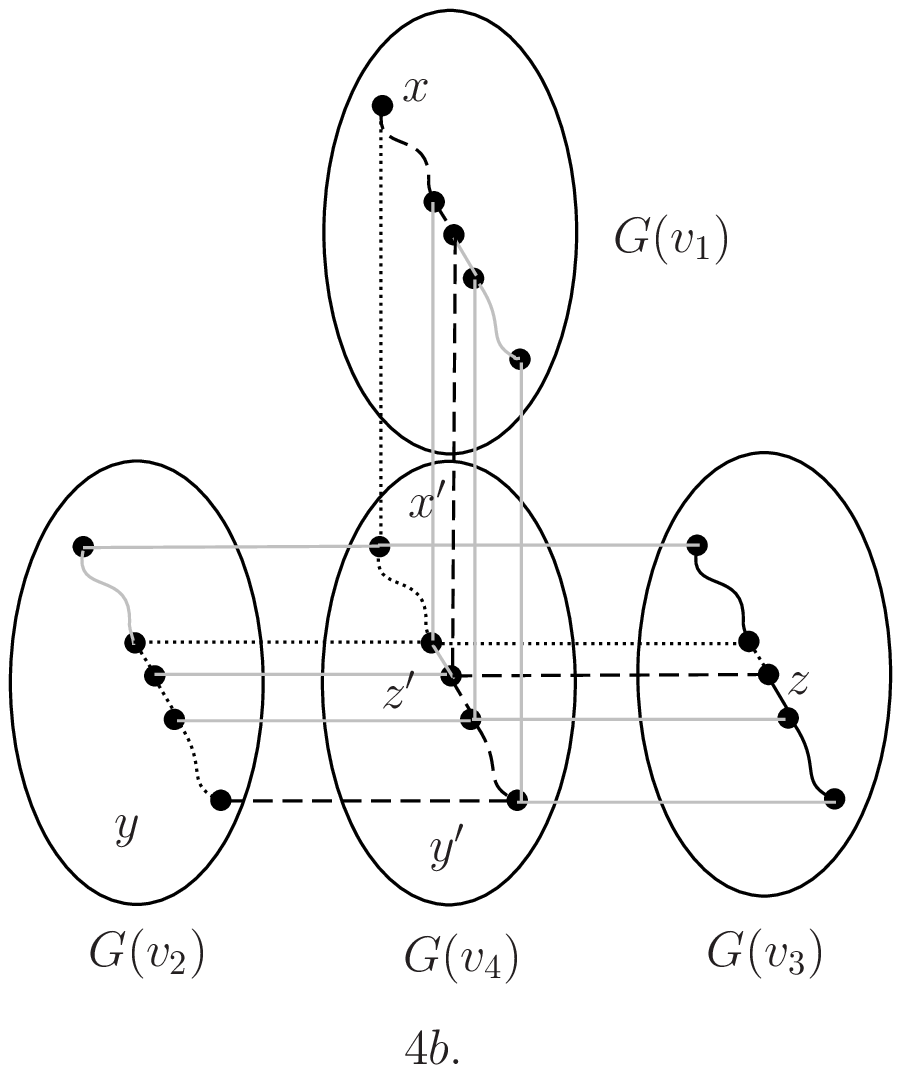}}
\scalebox{0.8}[0.7]{\includegraphics{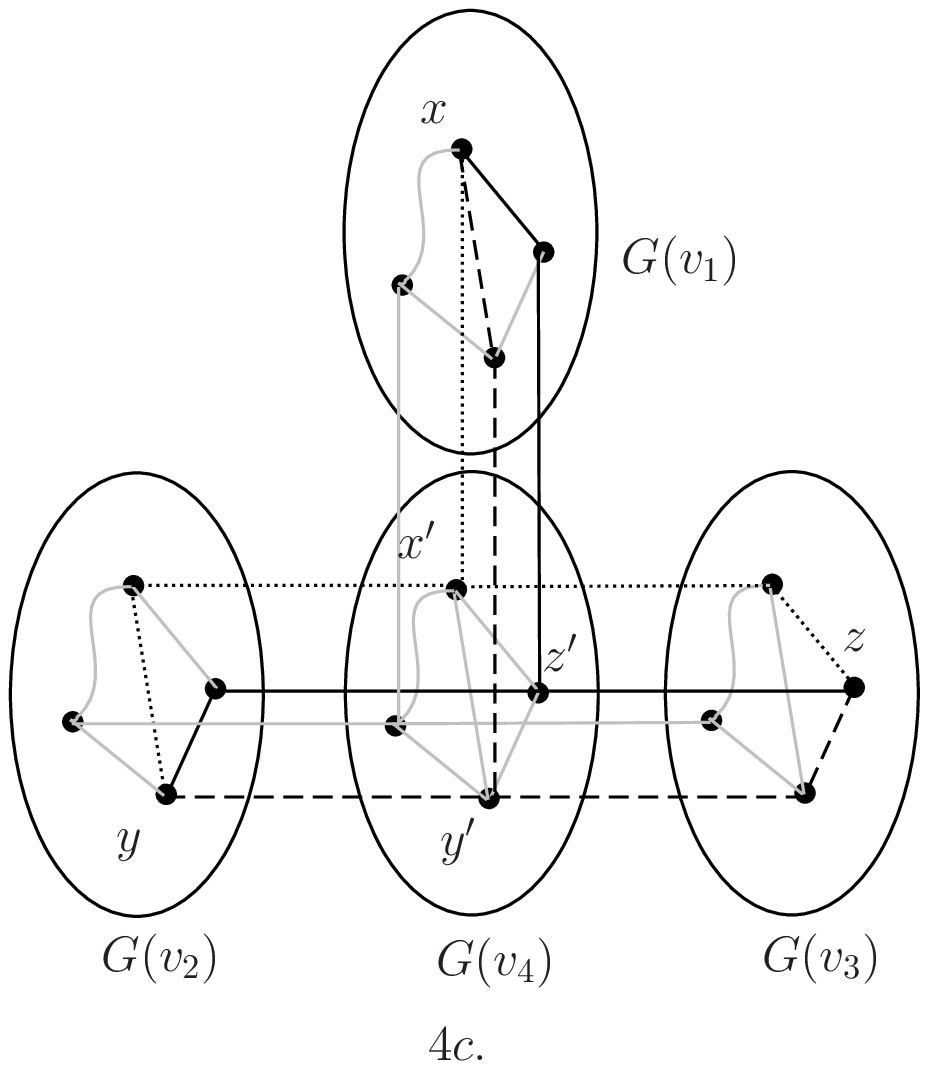}}

Figure~4. The edges (or paths) of a tree are shown by the same type
of lines.\\ The lightest lines stand for edges (or paths) not
contained in $T_i^*$.
\end{center}
\end{figure}

{\flushleft\textbf{Case 1.}} $E(T_i)\cap E(G(v_4)[S'])=\emptyset$
for each $i$, where $1\leq i\leq k$

Assume that $d_{T_i}(x')=d_{T_i}(y')=d_{T_i}(z')=1$ for $1\leq i\leq
k_1$ and set
$max\{d_{T_i}(x'),\linebreak[2]d_{T_i}(y'),d_{T_i}(z')\}=2$ for
$k_1+1\leq i\leq k$. If $k_1=k$, for $1\leq i\leq k-1$, assume that
$x', y'$ and $z'$ have neighbors $x_i,y_i,z_i$ in $T_i$,
respectively. (There maybe exist the same vertex in
$\{x_i,y_i,z_i\}$). Let ${T_i}^*$ be the tree obtained from $T_i$ by
adding edges $x_ix_i',x_i'x,y_iy_i',y_i'y,z_iz_i',z_i'z$ and
deleting $x',y',z'$, where $1\leq i \leq k-1$ and $x_i'\in
V(G(v_1)),y_i'\in V(G(v_2)),z_i'\in V(G(v_3))$ are the vertices
corresponding to $x_i,y_i,z_i$. Moreover, $T_k^*\cup T_{k+1}^*$ are
shown as in Figure~$4a$. Clearly, $T_1^*,T_2^*,\ldots,T_{k+1}^*$ are
$k+1$ internally disjoint $S$-trees.

Suppose $k_1\leq k-1$. For $1\leq i\leq k_1$, we construct $T_i^*$
the same as above. For $T_i$ with $k_1+1\leq i\leq k$, without loss
of generality, assume that
$d_{T_i}(z')=2,N_{T_i}(z')=\{z_{i,1},z_{i,2}\}, N_{T_i}(x')=\{x_i\},
N_{T_i}(y')=\{y_i\}$. Let ${T_i}^*$ be the tree obtained from $T_i$
by adding edges
$x_ix_i',x_i'x,y_iy_i',y_i'y,z_{i,1}z_{i,1}',z_{i,1}'z,z_{i,2}z_{i,2}',
z_{i,2}'z$ and deleting $x',y'z'$, where $k_1+1\leq i \leq k-1$ and
$z_{i,1}',z_{i,2}'\in V(G(v_3)),x_i'\in V(G(v_1)),y_i'\in V(G(v_2))$
are the vertices corresponding to $z_{i,1},z_{i,2},x_i,y_i$,
respectively. Moreover, $T_k^*$ and $T_{k+1}^*$ are as shown in
Figure~$4b$. Clearly, $T_1^*,T_2^*,\ldots,T_{k+1}^*$ are $k+1$
internally disjoint trees connecting $\{x,y,z\}$.

{\flushleft\textbf{Case 2.}} There exists some $T_i$ such that
$E(T_i)\cap E(G(v_4)[S'])\neq\emptyset$.

For a tree $T_i$ such that $E(T_i)\cap E(G(v_4)[S'])=\emptyset$, we
can construct ${T_i}^*$ similar to that in Case $1.1$.

If $E(T_{k-1})\cap E(G(v_4)[S'])=\emptyset$ and $E(T_k)\cap
E(G(v_4)[S'])\neq\emptyset$, $T_k^*$ and $T_{k+1}^*$ can be
constructed similar to those in Figure $4a.$ or $4b.$ (depending on
whether $d_{T_k}(x')=d_{T_k}(y')=d_{T_k}(z')=1$ or not).

If $E(T_{k-1})\cap E(G(v_4)[S'])\neq\emptyset$ and $E(T_k)\cap
E(G(v_4)[S'])\neq\emptyset$, then $T_{k-1}\cup T_{k}$ must have the
structures as shown in Figure $1e$. Without loss of generality, we
assume that $T_k$ is isomorphic to $P_3$ which has endpoints $x'$
and $y'$, and the only internal vertex $z'$. We can obtain trees
$T_{k-1}^*,T_{k}^*$ and $T_{k+1}^*$ as shown in Figure $4c$.
\end{proof}

\begin{lemma} If two of $x', y',z'$ are the same vertex in $G(v_4)$,
then there exist $k+1$ internally disjoint $S$-trees.
\end{lemma}
\begin{proof}
Without loss of generality, assume $y'=z'$. Since
$\kappa(G)>\kappa_3(G)=k$, by Menger's Theorem, there exist at least
$k+1$ internally disjoint {\em $x'$-$y'$ paths}
$P^1,P^2,\ldots,P^{k+1}$ in $G(v_4)$.

Assume that $y_i$ are the only neighbor of $y'$ in $T_i$, and $y_i'$
and $y_i''$ are the vertices corresponding to $y_i$ in $V(G(v_2))$
and $V(G(v_3))$, respectively, where $1\leq i \leq k+1$.

If $x'$ and $y'$ are nonadjacent, let $T_{i}$ be a tree obtained
from $P^i$ by adding $y_iy_i',y_i'y,\linebreak[3]y_i'y_i'',y_i''z$
and deleting $y'$; If $x'$ and $y'$ are adjacent, let $T_{i}$ be a
tree obtained from $P^i$ by adding $y_iy_i',y_i'y,y_i'y_i'',y_i''z$.
Since $G$ is a simple graph, there exists at most one path $P^i$
such that $x$ and $y'$ are adjacent on $P^i$. Thus,
$T_1,T_2,\ldots,T_{k+1}$ are $k+1$ internally disjoint $S$-trees.
\end{proof}

\begin{lemma} If $x',y',z'$ are the same vertex in $G(v_4)$,
then there exist $k+1$ internally disjoint $S$-trees.
\end{lemma}
\begin{proof}
Pick $w\in V(G(v_4))$ such that $x', w$ are distinct vertices in
$G(v_4)$. Since $\kappa(G)>\kappa_3(G)=k$, by Menger's Theorem,
there exist at least $k+1$ internally disjoint {\em $x'$-$w$ paths}
$P^1,P^2,\ldots,P^{k+1}$. Let $T_{i}$ be a tree obtained from $P^i$
by adding $x_ix_i',x_i'y,x_i'x_i'',x_i''z$ and deleting $y'$, where
$1\leq i \leq k+1$, $x_i$ is the only neighbor of $x$ in $P^i$, and
$x_i'$ and $x_i''$ are the vertices corresponding to $x_i$ in
$V(G(v_2))$ and $V(G(v_3))$, respectively. Clearly,
$T_1,T_2,\ldots,T_{k+1}$ are $k+1$ internally disjoint $S$-trees.
\end{proof}

\begin{remark}
We know that the bounds of $(i)$ and $(ii)$ in Theorem $3.1$ are
sharp by Example $3.1$ and Corollary~$3.1$.
\end{remark}

\begin{observation} The $k+1$ internally disjoint $S$-trees
consist of three kinds of edges --- the edges of original trees (or
paths), the edges corresponding the edges of original trees (or
paths) and two-type edges.
\end{observation}

\section{The Cartesian product of two general graphs}

\begin{observation}
Let $G$ and $H$ be two connected graphs, $x,y,z$ be three distinct
vertices in $H$, and $T_1,T_2,\ldots,T_k$ be $k$ internally disjoint
$\{x,y,z\}$-trees in $H$. Then $G\Box \bigcup_{i=1}^k
T_i=\bigcup_{i=1}^k (G\Box T_i)$ has the structure as shown in
Figure $5$. Moreover, $(G\Box T_i)\cap (G\Box T_j)=G(x)\cup G(y)\cup
G(z)$ for $i\neq j$. In order to show the structure of
$G\Box\bigcup_{i=1}^k T_i$ clearly, we take $k$ copies of $G(y)$,
and $k$ copies of $G(z)$. Note that, these $k$ copes of $G(y)$
(resp. $G(z)$) represent the same graph.
\end{observation}

\begin{figure}[h,t,b,p]
\begin{center}
\scalebox{0.6}[0.6]{\includegraphics{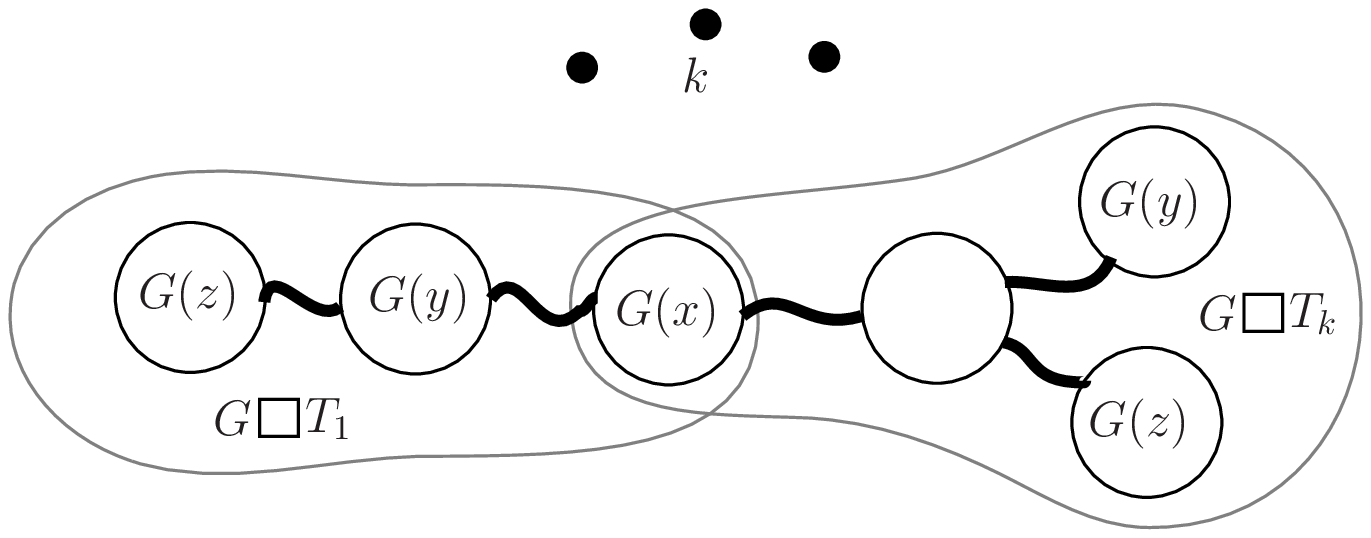}}

Figure~5. The structure of $G\Box\bigcup_{i=1}^k T_i$.
\end{center}
\end{figure}

\noindent{\bf Example 5.1.} Let $H$ be the complete graph with order
4. The structure of $G\Box (T_1 \cup T_2)$ are shown in Figure~$6$.

\begin{figure}[h]
\begin{center}
\scalebox{0.7}[0.7]{\includegraphics{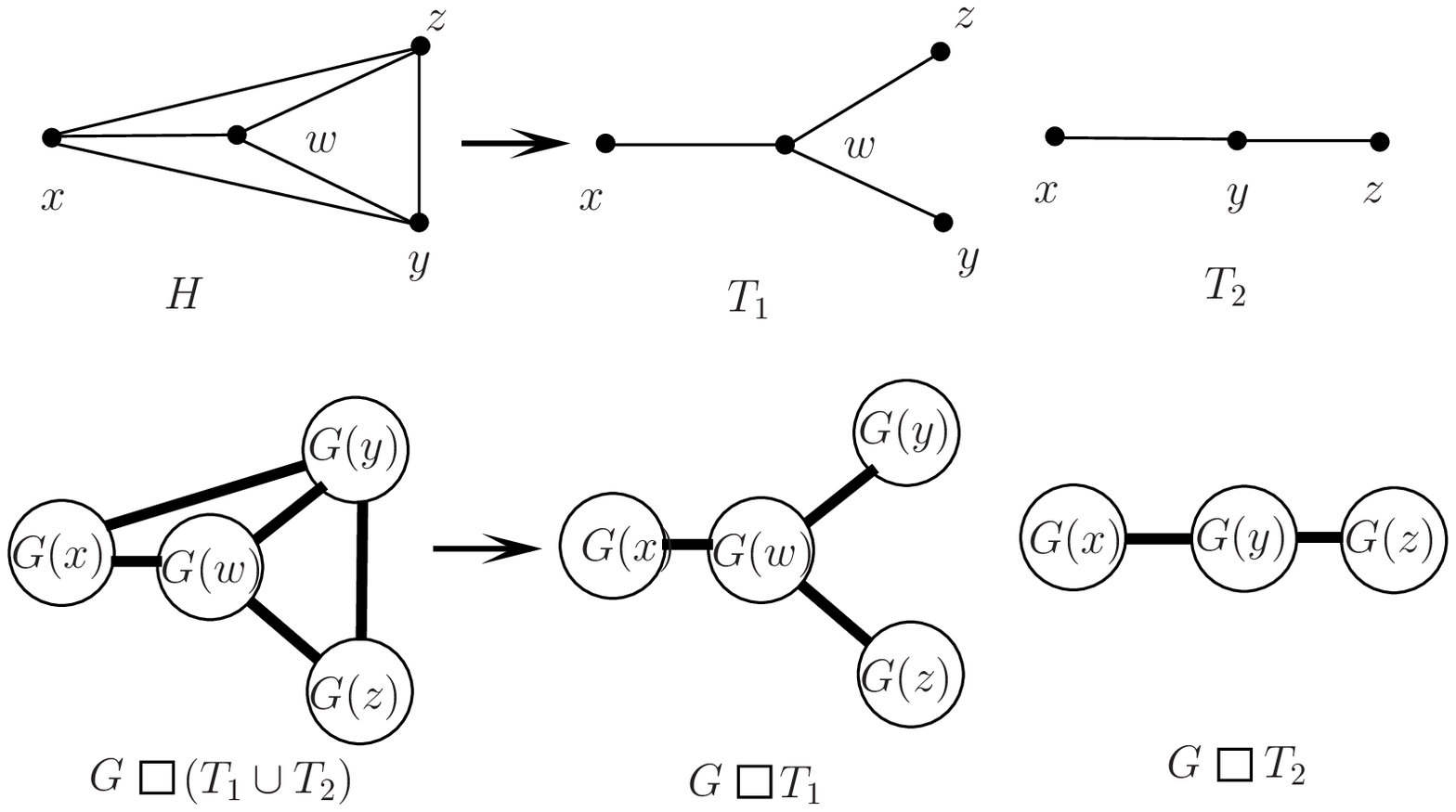}}

Figure~$6$. The structure of $G\Box (T_1\cup T_2)$.
\end{center}
\end{figure}

\begin{theorem} Let $G$ and $H$ be connected graphs such that
$\kappa_3(G)\geq\kappa_3(H)$.

$(i)$ If $\kappa(G)>\kappa_3(G)$, then $\kappa_3(G\Box H)\geq
\kappa_3(G)+\kappa_3(H)$. Moreover, the bound is sharp.

$(ii)$ If $\kappa(G)=\kappa_3(G)$, then $\kappa_3(G\Box H)\geq
\kappa_3(G)+\kappa_3(H)-1$. Moreover, the bound is sharp.

\end{theorem}
\begin{proof}
Since the proofs of $(i)$ and $(ii)$ are similar, we only show
$(ii)$. Without loss of generality, we set
$\kappa_3(G):=k,\kappa_3(H):=\ell$. It suffices to show that for any
$S=\{x,y,z\}\subseteq G\Box H$, there exist $k+\ell$ internally
disjoint $S$-trees. Assume $V(G)=\{u_1,u_2,\ldots,u_n\}$ and
$V(T)=\{v_1,v_2,\ldots,v_m\}$.

Let $x\in V(G(v_i)),y\in V(G(v_j)),z\in V(G(v_k))$ be three distinct
vertices in $G\Box H$. We consider the following three cases.

{\flushleft\textbf{Case 1.}} $i,j,k$ are distinct integers.

Without loss of generality, set $i=1,j=2,k=3$. Since
$\kappa_3(H)=\ell$, there exist $\ell$ internally disjoint
$\{v_1,v_2,v_3\}$-trees $T_i,\,1\leq i \leq \ell$, in $H$. We use
$G_i$ to denote $G\Box T_i$. By Observation~$5.1$, we know that
$G\Box \bigcup_{i=1}^\ell T_i=\bigcup_{i=1}^\ell G_i$ and $G_i\cap
G_j=G(v_1)\cup G(v_2)\cup G(v_3)$ for $i\neq j$. Let $y',z'$ be the
vertices corresponding to $y,z$ in $G(v_1)$, respectively. Consider
the following three subcases.

{\flushleft\textbf{Case 1.1.}} $x,y',z '$ are distinct vertices in
$G(v_1)$.

Since $\kappa_3(G(v_1))=k$, there exist $k$ internally disjoint
$\{x,y',z'\}$-trees $T_j',\,1\leq j \leq k$, in $G(v_1)$. Let $k_0,
k_1,\ldots,k_\ell$ be integers such that
$0=k_0<k_1<\cdots<k_\ell=k$. Similar to the proofs of Theorems~$3.1$
and~$4.1$, we can construct $k_i-k_{i-1}+1$ internally disjoint
$S$-trees $T_{i,j_i}, 1\leq j_i\leq k_i-k_{i-1}+1$, in
$(\bigcup_{j=k_{i-1}+1}^{k_i} T_j')\Box T_i$ for each $i$, where
$1\leq i \leq \ell$. By Observation~$4.1$, $T_{i,j_i}$ and
$T_{r,j_r}$ are internally disjoint for different integers $i,j$.
Thus $T_{i,j_i},1\leq i \leq \ell, 1\leq j_i \leq k_i-k_{i-1}+1$ are
$k+\ell$ internally disjoint $S$-trees.

Subcase~$1.2$ (exact two of $x,y',z '$ are the same vertex in
$G(v_i)$) and Subcase~$1.2$ (all of $x, y',z'$ are the same vertex
in $G(v_i)$) can be proved similarly, and the details are omitted.

Furthermore, Case~$2$ (exact two of $i,j,k$ are the same integer)
and Case~$3$ ($i=j=k$) can be proved similarly, and the details are
also omitted.

We now show that the bound of $(i)$ is sharp. Let $K_n$ be a
complete graph with $n$ vertices, and $P_m$ be a path with $m$
vertices, where $m\geq 3$. We have $P_m=1$, and $K_n=n-2$ by Theorem
$1.2$. It is easy to check that $K_n\Box P_m=n-2+1=n-1$ by Theorem
$1.3$. For $(ii)$, Example~$3.1$ is a sharp example.
\end{proof}

\end{CJK}
\end{document}